\documentclass[reqno]{amsart}
\usepackage{amssymb}
\usepackage{amsmath}
\usepackage{amsfonts}

\newtheorem{theorem}{Theorem}[section]
\theoremstyle{plain}

\newtheorem{lemma}{Lemma}[section]

\newtheorem{remark}{Remark}[section]

\numberwithin{equation}{section}

\begin{document}
\title[Global attractors for the plate equation]{Global attractors for the
plate equation with nonlocal nonlinearity in unbounded domains}
\author{{\small Zehra Arat }}
\address{{\small Department of Mathematics,} {\small Faculty of Science,
Hacettepe University, Beytepe 06800}, {\small Ankara, Turkey}}
\email{zarat@hacettepe.edu.tr}
\author{{\small Azer Khanmamedov}}
\email{azer@hacettepe.edu.tr}
\author{{\small Sema Simsek \ \ }}
\email{semasimsek@hacettepe.edu.tr}
\subjclass[2000]{35B41, 35G20, 37L30, 74K20}
\keywords{plate equation, global attractor}

\begin{abstract}
We consider the initial value problem for the semilinear plate equation with
nonlocal nonlinearity. We prove the existence of global attractor and then
establish the regularity and finite dimensionality of this attractor.
\end{abstract}

\maketitle

\section{Introduction}

\bigskip The main aim of this paper is to study the long time dynamics (in
terms of attractors) of the plate equation%
\begin{equation}
u_{tt}+\Delta ^{2}u+\alpha (x)u_{t}+\lambda u-f(\left\Vert \nabla u\left(
t\right) \right\Vert _{L^{2}\left(
\mathbb{R}
^{n}\right) })\Delta u=h\left( x\right) \text{, \ \ \ }(t,x)\in (0,\infty
)\times
\mathbb{R}
^{n}\text{,}  \tag{1.1}
\end{equation}%
with initial data%
\begin{equation}
u(0,x)=u_{0}(x)\text{, \ }u_{t}(0,x)=u_{1}(x)\text{, \ \ }x\in
\mathbb{R}
^{n}\text{,}  \tag{1.2}
\end{equation}%
where $\lambda >0,$ $h\in L^{2}\left(
\mathbb{R}
^{n}\right) $ and the functions $\alpha \left( \cdot \right) $, $f\left(
\cdot \right) $ satisfy the following conditions:%
\begin{equation}
\alpha \in L^{\infty }(%
\mathbb{R}
^{n})\text{, \ }\alpha (\cdot )\geq \alpha _{0}>0\text{ \ a.e. in }%
\mathbb{R}
^{n}\text{,}  \tag{1.3}
\end{equation}%
\begin{equation}
f\in C^{1}(%
\mathbb{R}
^{+})\text{, \ \ }f\left( z\right) \geq 0\text{, for all }z\in
\mathbb{R}
^{+}.  \tag{1.4}
\end{equation}%
\ \ \ By the semigroup theory, it is easy to show that under the conditions
(1.3) and (1.4), for every $\left( u_{0},u_{1}\right) \in H^{2}\left(
\mathbb{R}
^{n}\right) \times L^{2}\left(
\mathbb{R}
^{n}\right) $, the problem (1.1)-(1.2) has a unique weak solution $u\in
C\left( [0,\infty );H^{2}\left(
\mathbb{R}
^{n}\right) \right) \cap $ $C^{1}\left( [0,\infty );L^{2}\left(
\mathbb{R}
^{n}\right) \right) $, which depends continuously on the initial data and
satisfies the energy equality
\begin{equation*}
E\left( u\left( t\right) \right) +\frac{1}{2}F\left( \left\Vert \nabla
u\left( t\right) \right\Vert _{L^{2}\left(
\mathbb{R}
^{n}\right) }^{2}\right) -\int\limits_{%
\mathbb{R}
^{n}}h\left( x\right) u\left( t,x\right) dx+\int\limits_{s}^{t}\int\limits_{%
\mathbb{R}
^{n}}\alpha \left( x\right) \left\vert u_{t}\left( \tau ,x\right)
\right\vert ^{2}dxd\tau
\end{equation*}%
\begin{equation}
\text{ }=E\left( u\left( s\right) \right) +\frac{1}{2}F\left( \left\Vert
\nabla u\left( s\right) \right\Vert _{L^{2}\left(
\mathbb{R}
^{n}\right) }^{2}\right) -\int\limits_{%
\mathbb{R}
^{n}}h\left( x\right) u\left( s,x\right) dx\text{, \ }\forall t\geq s\geq 0%
\text{,}  \tag{1.5}
\end{equation}%
where $F\left( z\right) =\int\limits_{0}^{z}f\left( \sqrt{s}\right) ds$ for
all $z\in
\mathbb{R}
^{+}$ and $E\left( u\left( t\right) \right) =\frac{1}{2}\int\limits_{%
\mathbb{R}
^{n}}(\left\vert u_{t}\left( t,x\right) \right\vert ^{2}+\left\vert {\small %
\Delta u}\left( t,x\right) \right\vert ^{2}+\lambda \left\vert {\small u}%
\left( t,x\right) \right\vert ^{2})dx$. Moreover, if $\left(
u_{0},u_{1}\right) \in H^{4}\left(
\mathbb{R}
^{n}\right) \times H^{2}\left(
\mathbb{R}
^{n}\right) $, then $u$ is a strong solution from the class $C\left(
[0,\infty );H^{4}\left(
\mathbb{R}
^{n}\right) \right) $\newline
$\cap $ $C^{1}\left( [0,\infty );H^{2}\left(
\mathbb{R}
^{n}\right) \right) \cap $ $C^{2}\left( [0,\infty );L^{2}\left(
\mathbb{R}
^{n}\right) \right) $. Therefore, the problem (1.1)-(1.2) generates a
strongly continuous semigroup $\left\{ S\left( t\right) \right\} _{t\geq 0}$
in $H^{2}\left(
\mathbb{R}
^{n}\right) \times L^{2}\left(
\mathbb{R}
^{n}\right) $ by the formula $\left( u\left( t\right) ,u_{t}\left( t\right)
\right) =S\left( t\right) (u_{0},u_{1})$, where $u\left( t,x\right) $ is a
weak solution of (1.1)-(1.2) with the initial data $\left(
u_{0},u_{1}\right) $. By (1.4) and (1.5), we have the inequality%
\begin{equation}
E\left( u\left( t\right) \right) +\int\limits_{0}^{t}\int\limits_{%
\mathbb{R}
^{n}}\alpha \left( x\right) \left\vert u_{t}\left( \tau ,x\right)
\right\vert ^{2}dxd\tau \leq c\left( \left\Vert \left( u_{0},u_{1}\right)
\right\Vert _{H^{2}\left(
\mathbb{R}
^{n}\right) \times L^{2}\left(
\mathbb{R}
^{n}\right) }\right) \text{, \ \ }\forall t\geq 0\text{,}  \tag{1.6}
\end{equation}%
which implies the boundedness of $\left\{ S\left( t\right) \right\} _{t\geq
0}$ in $H^{2}\left(
\mathbb{R}
^{n}\right) \times L^{2}\left(
\mathbb{R}
^{n}\right) $, where $c:%
\mathbb{R}
^{+}\rightarrow
\mathbb{R}
^{+}$ is a nondecreasing function.

The problem of investigating the asymptotic behavior of evolution equations
modeling many physical phenomena has been attracting more attention over the
last few decades. It is well known that the asymptotic behavior of these
equations can be described by means of the attractors. The attractors for
plate equations has been one of the intensively studied topic in recent
years. We refer to [1-10] for attractors of plate equations with local and
nonlocal nonlinearities in bounded domains. In the case of unbounded
domains, there are obstacles in applying the methods given for bounded
domains due to the lack of Sobolev compact embedding theorems. So as to
handle these obstacles, the authors of [11-14] established the uniform tail
estimates for the plate equations with local nonlinearities.

The situation becomes more difficult when the domain is unbounded and the
equation includes nonlocal nonlinearity, for example $f(\left\Vert \nabla
u\left( t\right) \right\Vert _{L^{2}\left(
\mathbb{R}
^{n}\right) })\Delta u(t)$ as in the case of equation (1.1). When $%
f(s)=s^{2} $, this nonlocal term becomes famous Berger nonlinearity (see
[15]). In the unbounded domain case, the operator $\mathcal{F}(u):=$ $%
f(\left\Vert \nabla u\right\Vert _{L^{2}\left(
\mathbb{R}
^{n}\right) })\Delta u$ which is determined by the nonlocal term mentioned
above, besides being not compact, is not also weakly continuous from $%
H^{2}\left(
\mathbb{R}
^{n}\right) $ to $L^{2}\left(
\mathbb{R}
^{n}\right) $. So, in order to establish the asymptotic compactness which is
necessary for the existence of the global attractor, we are not able to
apply either the standard splitting method or the energy method devised in
[16]. To overcome these difficulties, we apply compensated compactness
method introduced in [17] and prove the asymptotic compactness (see Lemma
2.2) which, together with the presence of the strict Lyapunov function,
leads to the existence of a global attractor. Then, by using the invariance
of the global attractor and the structural property of the set of stationary
points, we establish the regularity (see\ Theorem 3.1) and consequently, the
finite dimensionality (see Theorem 4.1) of the global attractor.

Our main result is as follows:

\begin{theorem}
Under conditions (1.3) and (1.4) the semigroup $\left\{ S\left( t\right)
\right\} _{t\geq 0}$ generated by the problem (1.1)-(1.2) possesses a global
attractor $\mathcal{A}$ in $H^{2}\left(
\mathbb{R}
^{n}\right) \times L^{2}\left(
\mathbb{R}
^{n}\right) $ and $\mathcal{A=M}^{u}\left( \mathcal{N}\right) $. Here $%
\mathcal{M}^{u}\left( \mathcal{N}\right) $ is unstable manifold emanating
from the set of stationary points $\mathcal{N}$ (for definition, see [18,
p.359]). Moreover, the global attractor $\mathcal{A}$ is bounded in $%
H^{4}\left(
\mathbb{R}
^{n}\right) \times H^{2}\left(
\mathbb{R}
^{n}\right) $ and it has finite fractal dimension.

\begin{remark}
We note that by using the method of this paper, one can prove the existence,
regularity and finite dimensionality of the global attractor for the initial
boundary value problem%
\begin{equation}
\left\{
\begin{array}{c}
u_{tt}+\Delta ^{2}u+\alpha (x)u_{t}+\lambda u-f(\left\Vert \nabla u\left(
t\right) \right\Vert _{L^{2}\left( \Omega \right) })\Delta u=h\left(
x\right) \text{, \ \ }(t,x)\in (0,\infty )\times \Omega \text{,} \\
u(t,x)=\frac{\partial }{\partial \nu }u(t,x)=0\text{, \ \ \ \ \ \ \ \ \ \ \
\ \ \ \ \ \ \ \ \ \ \ \ \ \ \ \ \ \ \ \ \ \ \ \ \ \ \ \ \ \ \ }(t,x)\in
(0,\infty )\times \partial \Omega \text{,} \\
u(0,x)=u_{0}(x)\text{, \ \ \ \ }u_{t}(0,x)=u_{1}(x)\text{, \ \ \ \ \ \ \ \ \
\ \ \ \ \ \ \ \ \ \ \ \ \ \ \ \ \ \ \ \ \ \ \ \ \ \ \ \ \ \ \ \ \ \ }x\in
\Omega \text{,}%
\end{array}%
\right.   \tag{1.7}
\end{equation}%
where $\Omega \subset
\mathbb{R}
^{n}$ is an unbounded domain with smooth boundary, $\nu $ is outer unit
normal vector, $\lambda >0$, $h\in L^{2}\left( \Omega \right) $, the
function $f(\cdot )$ satisfies the condition (1.4) and the damping
coefficient $\alpha (\cdot )$ satisfies the following conditions%
\begin{equation*}
\alpha \in L^{\infty }(\Omega )\text{, \ \ }\alpha (\cdot )\geq \alpha
_{0}>0,\text{ a.e. in }\Omega \text{.}
\end{equation*}
\end{remark}
\end{theorem}

\begin{remark}
We also note that we critically use the strict positivity of $\alpha (\cdot
) $ (see (1.3)) in the proof of asymptotic compactness of the semigroup $%
\left\{ S\left( t\right) \right\} _{t\geq 0}$ (see Lemma 2.2). In the case
when the function $\alpha (\cdot )$ is not strictly positive, for example if
$\alpha (\cdot )$ vanishes in a set of positive measure, our method is not
applicable. Another obstacle in this case is related to unique continuation
of solutions which is important for the construction of a strict Lyapunov
function. To the best of our knowledge the unique continuation of solutions
for the equations (1.1) and (1.7)$_{1}$-(1.7)$_{2}$ is also an open
question. Thus, in the case when $\alpha (\cdot )$ vanishes in a set of
positive measure, the questions about long time dynamics of (1.1)-(1.2) and
(1.7), in terms of attractors, are completely open.
\end{remark}

\section{Existence of the global attractor}

In this section, we will show the existence of the global attractor. To this
end, we first prove the following lemma.

\begin{lemma}
Let the conditions (1.3) and (1.4) hold. Also, assume that\ the sequence $%
\left\{ v_{m}\right\} _{m=1}^{\infty }$ is bounded in $L^{\infty }\left(
0,T;H^{2}\left(
\mathbb{R}
^{n}\right) \right) \cap W^{1,\infty }\left( 0,T;L^{2}\left(
\mathbb{R}
^{n}\right) \right) $ and the sequence $\left\{ \left\Vert \nabla
v_{m}\left( t\right) \right\Vert _{L^{2}\left(
\mathbb{R}
^{n}\right) }\right\} _{m=1}^{\infty }$ is convergent, for all $t\in \lbrack
0,T]$. Then, for all $\gamma >0$, there exists some $c_{\gamma }>0$ such that%
\begin{equation*}
\int\limits_{0}^{t}\int\limits_{%
\mathbb{R}
^{n}}\tau \left( f\left( \left\Vert \nabla v_{m}\left( \tau \right)
\right\Vert _{L^{2}\left(
\mathbb{R}
^{n}\right) }\right) \Delta v_{m}(\tau ,x)-f\left( \left\Vert \nabla
v_{l}\left( \tau \right) \right\Vert _{L^{2}\left(
\mathbb{R}
^{n}\right) }\right) \Delta v_{l}(\tau ,x)\right) \left( v_{mt}\left( \tau
,x\right) -v_{lt}\left( \tau ,x\right) \right) dxd\tau
\end{equation*}%
\begin{equation*}
\leq \gamma \int\limits_{0}^{t}\tau E\left( v_{m}\left( \tau \right)
-v_{l}\left( \tau \right) \right) d\tau +c_{\gamma
}\int\limits_{0}^{t}E\left( v_{m}\left( \tau \right) -v_{l}\left( \tau
\right) \right) d\tau
\end{equation*}%
\begin{equation*}
+c_{\gamma }\int\limits_{0}^{t}\tau E\left( v_{m}\left( \tau \right)
-v_{l}\left( \tau \right) \right) \left\Vert v_{mt}\left( \tau \right)
\right\Vert _{L^{2}\left(
\mathbb{R}
^{n}\right) }^{2}d\tau +K^{m,l}(t)\text{,}
\end{equation*}%
for all $t\in \lbrack 0,T]$, where $K^{m,l}\in C[0,T]$ \ and \ $\underset{%
m\rightarrow \infty }{\lim \sup }$ $\underset{l\rightarrow \infty }{\lim
\sup }\left\Vert K^{m,l}\right\Vert _{C[0,T]}=0$.
\end{lemma}

\begin{proof}
Firstly, we have%
\begin{equation*}
\int\limits_{0}^{t}\int\limits_{%
\mathbb{R}
^{n}}\tau \left( f\left( \left\Vert \nabla v_{m}\left( \tau \right)
\right\Vert _{L^{2}\left(
\mathbb{R}
^{n}\right) }\right) \Delta v_{m}(\tau ,x)-f\left( \left\Vert \nabla
v_{l}\left( \tau \right) \right\Vert _{L^{2}\left(
\mathbb{R}
^{n}\right) }\right) \Delta v_{l}(\tau ,x)\right) \left( v_{mt}\left( \tau
,x\right) -v_{lt}\left( \tau ,x\right) \right) dxd\tau
\end{equation*}%
\begin{equation}
=-\int\limits_{0}^{t}\tau f\left( \left\Vert \nabla v_{m}\left( \tau \right)
\right\Vert _{L^{2}\left(
\mathbb{R}
^{n}\right) }\right) \frac{d}{d\tau }\left( \left\Vert \nabla v_{m}(\tau
)-\nabla v_{l}(\tau )\right\Vert _{L^{2}\left(
\mathbb{R}
^{n}\right) }^{2}\right) dt+K^{m,l}\left( t\right) ,  \tag{2.1}
\end{equation}%
where%
\begin{equation*}
K^{m,l}(t):=\int\limits_{0}^{t}\int\limits_{%
\mathbb{R}
^{n}}\tau \left( f\left( \left\Vert \nabla v_{m}\left( \tau \right)
\right\Vert _{L^{2}\left(
\mathbb{R}
^{n}\right) }\right) -f\left( \left\Vert \nabla v_{l}\left( \tau \right)
\right\Vert _{L^{2}\left(
\mathbb{R}
^{n}\right) }\right) \right) \Delta v_{l}\left( \tau ,x\right) \left(
v_{mt}\left( \tau ,x\right) -v_{lt}\left( \tau ,x\right) \right) dxd\tau
\text{.}
\end{equation*}%
By the conditions of the lemma, we obtain%
\begin{equation}
\underset{m\rightarrow \infty }{\text{ }\lim \sup \text{ }}\underset{%
l\rightarrow \infty }{\lim \sup }\left\Vert K^{m,l}\right\Vert _{C[0,T]}=0%
\text{.}  \tag{2.2}
\end{equation}%
Now, let us estimate the first term on the right side of (2.1). For any $%
\varepsilon >0$, by integration by parts, we have%
\begin{equation*}
\int\limits_{0}^{t}\tau f\left( \left\Vert \nabla v_{m}\left( \tau \right)
\right\Vert _{L^{2}\left(
\mathbb{R}
^{n}\right) }\right) \frac{d}{d\tau }\left( \left\Vert \nabla v_{m}(\tau
)-\nabla v_{l}(\tau )\right\Vert _{L^{2}\left(
\mathbb{R}
^{n}\right) }^{2}\right) d\tau
\end{equation*}%
\begin{equation*}
=\int\limits_{0}^{t}\tau \left( f\left( \left\Vert \nabla v_{m}\left( \tau
\right) \right\Vert _{L^{2}\left(
\mathbb{R}
^{n}\right) }\right) -f\left( \varepsilon \right) \right) \frac{d}{d\tau }%
\left( \left\Vert \nabla v_{m}(\tau )-\nabla v_{l}(\tau )\right\Vert
_{L^{2}\left(
\mathbb{R}
^{n}\right) }^{2}\right) d\tau
\end{equation*}%
\begin{equation*}
+f\left( \varepsilon \right) \int\limits_{0}^{t}\tau \frac{d}{d\tau }\left(
\left\Vert \nabla v_{m}(\tau )-\nabla v_{l}(\tau )\right\Vert _{L^{2}\left(
\mathbb{R}
^{n}\right) }^{2}\right) d\tau
\end{equation*}%
\begin{equation*}
=\int\limits_{A_{1,\varepsilon }^{m}(t)}\tau \left( f\left( \left\Vert
\nabla v_{m}\left( \tau \right) \right\Vert _{L^{2}\left(
\mathbb{R}
^{n}\right) }\right) -f\left( \varepsilon \right) \right) \frac{d}{d\tau }%
\left( \left\Vert \nabla v_{m}(\tau )-\nabla v_{l}(\tau )\right\Vert
_{L^{2}\left(
\mathbb{R}
^{n}\right) }^{2}\right) d\tau
\end{equation*}%
\begin{equation*}
+\int\limits_{A_{2,\varepsilon }^{m}(t)}\tau \left( f\left( \left\Vert
\nabla v_{m}\left( \tau \right) \right\Vert _{L^{2}\left(
\mathbb{R}
^{n}\right) }\right) -f\left( \varepsilon \right) \right) \frac{d}{d\tau }%
\left( \left\Vert \nabla v_{m}(\tau )-\nabla v_{l}(\tau )\right\Vert
_{L^{2}\left(
\mathbb{R}
^{n}\right) }^{2}\right) d\tau
\end{equation*}%
\begin{equation}
+tf\left( \varepsilon \right) \left\Vert \nabla v_{m}(t)-\nabla
v_{l}(t)\right\Vert _{L^{2}\left(
\mathbb{R}
^{n}\right) }^{2}-f\left( \varepsilon \right) \int\limits_{0}^{t}\left\Vert
\nabla v_{m}(\tau )-\nabla v_{l}(\tau )\right\Vert _{L^{2}\left(
\mathbb{R}
^{n}\right) }^{2}d\tau ,  \tag{2.3}
\end{equation}%
where%
\begin{equation*}
A_{1,\varepsilon }^{m}(t):=\left\{ \tau \in (0,t):\left\Vert \nabla
v_{m}(\tau )\right\Vert _{L^{2}\left(
\mathbb{R}
^{n}\right) }\leq \varepsilon \right\} \text{,}
\end{equation*}%
\begin{equation*}
A_{2,\varepsilon }^{m}(t):=\left\{ \tau \in (0,t):\left\Vert \nabla
v_{m}(\tau )\right\Vert _{L^{2}\left(
\mathbb{R}
^{n}\right) }>\varepsilon \right\} \text{.}
\end{equation*}%
Let us estimate the term%
\begin{equation*}
\int\limits_{A_{2,\varepsilon }^{m}(t)}\tau \left( f\left( \left\Vert \nabla
v_{m}\left( \tau \right) \right\Vert _{L^{2}\left(
\mathbb{R}
^{n}\right) }\right) -f\left( \varepsilon \right) \right) \frac{d}{d\tau }%
\left( \left\Vert \nabla v_{m}(\tau )-\nabla v_{l}(\tau )\right\Vert
_{L^{2}\left(
\mathbb{R}
^{n}\right) }^{2}\right) d\tau \text{.}
\end{equation*}%
It is enough to consider the case $A_{2,\varepsilon }^{m}(t)$ is nonempty.
Since, thanks to $v_{m}\in C([0,T];H^{1}\left(
\mathbb{R}
^{n}\right) )$, the set $A_{2,\varepsilon }^{m}(t)$ is open, it can be shown
that it is a countable union of disjoint open intervals $\left\{ (t_{k},%
\widetilde{t}_{k}\right\} _{k=1}^{\infty }$ (see, for example [19, p. 39]).\
Then%
\begin{equation*}
\int\limits_{A_{2,\varepsilon }^{m}(t)}\tau \left( f\left( \left\Vert \nabla
v_{m}\left( \tau \right) \right\Vert _{L^{2}\left(
\mathbb{R}
^{n}\right) }\right) -f\left( \varepsilon \right) \right) \frac{d}{d\tau }%
\left( \left\Vert \nabla v_{m}(\tau )-\nabla v_{l}(\tau )\right\Vert
_{L^{2}\left(
\mathbb{R}
^{n}\right) }^{2}\right) d\tau
\end{equation*}%
\begin{equation*}
=\sum_{k=0}^{\infty }\int\limits_{t_{k}}^{\widetilde{t}_{k}}\tau \left(
f\left( \left\Vert \nabla v_{m}\left( \tau \right) \right\Vert _{L^{2}\left(
\mathbb{R}
^{n}\right) }\right) -f\left( \varepsilon \right) \right) \frac{d}{d\tau }%
\left( \left\Vert \nabla v_{m}(\tau )-\nabla v_{l}(\tau )\right\Vert
_{L^{2}\left(
\mathbb{R}
^{n}\right) }^{2}\right) d\tau .
\end{equation*}%
If $\left\Vert \nabla v_{m}\left( t\right) \right\Vert _{L^{2}\left(
\mathbb{R}
^{n}\right) }\leq \varepsilon $, then by continuity of $f$, we have%
\begin{equation*}
f\left( \left\Vert \nabla v_{m}\left( t_{k}\right) \right\Vert _{L^{2}\left(
\mathbb{R}
^{n}\right) }\right) =f\left( \varepsilon \right) \text{, \ }k=1,2,...\text{
,}
\end{equation*}%
\begin{equation*}
f\left( \left\Vert \nabla v_{m}\left( \widetilde{t}_{k}\right) \right\Vert
_{L^{2}\left(
\mathbb{R}
^{n}\right) }\right) =f\left( \varepsilon \right) \text{, \ }k=1,2,...\text{
.}
\end{equation*}%
Hence, by integration by parts, we obtain%
\begin{equation*}
\int\limits_{A_{2,\varepsilon }^{m}(t)}\tau \left( f\left( \left\Vert \nabla
v_{m}\left( \tau \right) \right\Vert _{L^{2}\left(
\mathbb{R}
^{n}\right) }\right) -f\left( \varepsilon \right) \right) \frac{d}{d\tau }%
\left( \left\Vert \nabla v_{m}(\tau )-\nabla v_{l}(\tau )\right\Vert
_{L^{2}\left(
\mathbb{R}
^{n}\right) }^{2}\right) d\tau
\end{equation*}%
\begin{equation*}
=\sum_{k=1}^{\infty }\widetilde{t}_{k}\left( f\left( \left\Vert \nabla
v_{m}\left( \widetilde{t}_{k}\right) \right\Vert _{L^{2}\left(
\mathbb{R}
^{n}\right) }\right) -f\left( \varepsilon \right) \right) \left\Vert \nabla
v_{m}(\widetilde{t}_{k})-\nabla v_{l}(\widetilde{t}_{k})\right\Vert
_{L^{2}\left(
\mathbb{R}
^{n}\right) }^{2}
\end{equation*}%
\begin{equation*}
-\sum_{k=1}^{\infty }t_{k}\left( f\left( \left\Vert \nabla v_{m}\left(
t_{k}\right) \right\Vert _{L^{2}\left(
\mathbb{R}
^{n}\right) }\right) -f\left( \varepsilon \right) \right) \left\Vert \nabla
v_{m}(t_{k})-\nabla v_{l}(t_{k})\right\Vert _{L^{2}\left(
\mathbb{R}
^{n}\right) }^{2}
\end{equation*}%
\begin{equation*}
-\sum_{k=1}^{\infty }\int\limits_{t_{k}}^{\widetilde{t}_{k}}\left( f\left(
\left\Vert \nabla v_{m}\left( \tau \right) \right\Vert _{L^{2}\left(
\mathbb{R}
^{n}\right) }\right) -f\left( \varepsilon \right) \right) \left\Vert \nabla
v_{m}(\tau )-\nabla v_{l}(\tau )\right\Vert _{L^{2}\left(
\mathbb{R}
^{n}\right) }^{2}d\tau
\end{equation*}%
\begin{equation*}
+\sum_{k=1}^{\infty }\int\limits_{t_{k}}^{\widetilde{t}_{k}}\tau \frac{%
f^{\prime }\left( \left\Vert \nabla v_{m}\left( \tau \right) \right\Vert
_{L^{2}\left(
\mathbb{R}
^{n}\right) }\right) }{\left\Vert \nabla v_{m}\left( \tau \right)
\right\Vert _{L^{2}\left(
\mathbb{R}
^{n}\right) }}\left\langle \Delta v_{m}\left( \tau \right) ,v_{mt}\left(
\tau \right) \right\rangle _{L^{2}\left(
\mathbb{R}
^{n}\right) }\left\Vert \nabla v_{m}(\tau )-\nabla v_{l}(\tau )\right\Vert
_{L^{2}\left(
\mathbb{R}
^{n}\right) }^{2}d\tau
\end{equation*}%
\begin{equation*}
=-\int\limits_{A_{2,\varepsilon }^{m}(t)}\left( f\left( \left\Vert \nabla
v_{m}\left( \tau \right) \right\Vert _{L^{2}\left(
\mathbb{R}
^{n}\right) }\right) -f\left( \varepsilon \right) \right) \left\Vert \nabla
v_{m}(\tau )-\nabla v_{l}(\tau )\right\Vert _{L^{2}\left(
\mathbb{R}
^{n}\right) }^{2}d\tau
\end{equation*}%
\begin{equation}
+\int\limits_{A_{2,\varepsilon }^{m}(t)}\tau \frac{f^{\prime }\left(
\left\Vert \nabla v_{m}\left( \tau \right) \right\Vert _{L^{2}\left(
\mathbb{R}
^{n}\right) }\right) }{\left\Vert \nabla v_{m}\left( \tau \right)
\right\Vert _{L^{2}\left(
\mathbb{R}
^{n}\right) }}\left\langle \Delta v_{m}\left( \tau \right) ,v_{mt}\left(
\tau \right) \right\rangle _{L^{2}\left(
\mathbb{R}
^{n}\right) }\left\Vert \nabla v_{m}(\tau )-\nabla v_{l}(\tau )\right\Vert
_{L^{2}\left(
\mathbb{R}
^{n}\right) }^{2}d\tau ,  \tag{2.4}
\end{equation}%
where $\left\langle \cdot ,\cdot \right\rangle _{L^{2}\left(
\mathbb{R}
^{n}\right) }$ is an inner product in $L^{2}\left(
\mathbb{R}
^{n}\right) $. \ If $\left\Vert \nabla v_{m}\left( t\right) \right\Vert
_{L^{2}\left(
\mathbb{R}
^{n}\right) }>\varepsilon $, then the maximal element of $\left\{ \widetilde{%
t}_{k}:k=1,2,...\right\} $ is equal to $t$ and consequently, we have%
\begin{equation*}
\int\limits_{A_{2,\varepsilon }^{m}(t)}\tau \left( f\left( \left\Vert \nabla
v_{m}\left( \tau \right) \right\Vert _{L^{2}\left(
\mathbb{R}
^{n}\right) }\right) -f\left( \varepsilon \right) \right) \frac{d}{d\tau }%
\left( \left\Vert \nabla v_{m}(\tau )-\nabla v_{l}(\tau )\right\Vert
_{L^{2}\left(
\mathbb{R}
^{n}\right) }^{2}\right) d\tau
\end{equation*}%
\begin{equation*}
=\sum_{k=1}^{\infty }\widetilde{t}_{k}\left( f\left( \left\Vert \nabla
v_{m}\left( \widetilde{t}_{k}\right) \right\Vert _{L^{2}\left(
\mathbb{R}
^{n}\right) }\right) -f\left( \varepsilon \right) \right) \left\Vert \nabla
v_{m}(\widetilde{t}_{k})-\nabla v_{l}(\widetilde{t}_{k})\right\Vert
_{L^{2}\left(
\mathbb{R}
^{n}\right) }^{2}
\end{equation*}%
\begin{equation*}
-\sum_{k=1}^{\infty }t_{k}\left( f\left( \left\Vert \nabla v_{m}\left(
t_{k}\right) \right\Vert _{L^{2}\left(
\mathbb{R}
^{n}\right) }\right) -f\left( \varepsilon \right) \right) \left\Vert \nabla
v_{m}(t_{k})-\nabla v_{l}(t_{k})\right\Vert _{L^{2}\left(
\mathbb{R}
^{n}\right) }^{2}
\end{equation*}%
\begin{equation*}
-\sum_{k=1}^{\infty }\int\limits_{t_{k}}^{\widetilde{t}_{k}}\left( f\left(
\left\Vert \nabla v_{m}\left( \tau \right) \right\Vert _{L^{2}\left(
\mathbb{R}
^{n}\right) }\right) -f\left( \varepsilon \right) \right) \left\Vert \nabla
v_{m}(\tau )-\nabla v_{l}(\tau )\right\Vert _{L^{2}\left(
\mathbb{R}
^{n}\right) }^{2}d\tau
\end{equation*}%
\begin{equation*}
+\sum_{k=1}^{\infty }\int\limits_{t_{k}}^{\widetilde{t}_{k}}\tau \frac{%
f^{\prime }\left( \left\Vert \nabla v_{m}\left( \tau \right) \right\Vert
_{L^{2}\left(
\mathbb{R}
^{n}\right) }\right) }{\left\Vert \nabla v_{m}\left( \tau \right)
\right\Vert _{L^{2}\left(
\mathbb{R}
^{n}\right) }}\left\langle \Delta v_{m}\left( \tau \right) ,v_{mt}\left(
\tau \right) \right\rangle _{L^{2}\left(
\mathbb{R}
^{n}\right) }\left\Vert \nabla v_{m}(\tau )-\nabla v_{l}(\tau )\right\Vert
_{L^{2}\left(
\mathbb{R}
^{n}\right) }^{2}d\tau
\end{equation*}%
\begin{equation*}
=t\left( f\left( \left\Vert \nabla v_{m}\left( t\right) \right\Vert
_{L^{2}\left(
\mathbb{R}
^{n}\right) }\right) -f\left( \varepsilon \right) \right) \left\Vert \nabla
v_{m}(t)-\nabla v_{l}(t)\right\Vert _{L^{2}\left(
\mathbb{R}
^{n}\right) }^{2}
\end{equation*}%
\begin{equation*}
-\int\limits_{A_{2,\varepsilon }^{m}(t)}\left( f\left( \left\Vert \nabla
v_{m}\left( \tau \right) \right\Vert _{L^{2}\left(
\mathbb{R}
^{n}\right) }\right) -f\left( \varepsilon \right) \right) \left\Vert \nabla
v_{m}(\tau )-\nabla v_{l}(\tau )\right\Vert _{L^{2}\left(
\mathbb{R}
^{n}\right) }^{2}d\tau
\end{equation*}%
\begin{equation}
+\int\limits_{A_{2,\varepsilon }^{m}(t)}\tau \frac{f^{\prime }\left(
\left\Vert \nabla v_{m}\left( \tau \right) \right\Vert _{L^{2}\left(
\mathbb{R}
^{n}\right) }\right) }{\left\Vert \nabla v_{m}\left( \tau \right)
\right\Vert _{L^{2}\left(
\mathbb{R}
^{n}\right) }}\left\langle \Delta v_{m}\left( \tau \right) ,v_{mt}\left(
\tau \right) \right\rangle _{L^{2}\left(
\mathbb{R}
^{n}\right) }\left\Vert \nabla v_{m}(\tau )-\nabla v_{l}(\tau )\right\Vert
_{L^{2}\left(
\mathbb{R}
^{n}\right) }^{2}d\tau .  \tag{2.5}
\end{equation}%
Hence, by using (1.4), (2.4) and (2.5) in (2.3), we find%
\begin{equation*}
-\int\limits_{0}^{t}\tau f\left( \left\Vert \nabla v_{m}\left( \tau \right)
\right\Vert _{L^{2}\left(
\mathbb{R}
^{n}\right) }\right) \frac{d}{d\tau }\left( \left\Vert \nabla v_{m}(\tau
)-\nabla v_{l}(\tau )\right\Vert _{L^{2}\left(
\mathbb{R}
^{n}\right) }^{2}\right) d\tau
\end{equation*}%
\begin{equation*}
\leq f\left( \varepsilon \right) \int\limits_{0}^{t}\left\Vert \nabla
v_{m}(\tau )-\nabla v_{l}(\tau )\right\Vert _{L^{2}\left(
\mathbb{R}
^{n}\right) }^{2}d\tau +\widehat{c}_{1}\int\limits_{0}^{t}\left\Vert \nabla
v_{m}(\tau )-\nabla v_{l}(\tau )\right\Vert _{L^{2}\left(
\mathbb{R}
^{n}\right) }^{2}d\tau
\end{equation*}%
\begin{equation*}
+2\max_{s_{1},s_{2}\epsilon \left[ 0,\varepsilon \right] }\left\vert f\left(
s_{1}\right) -f\left( s_{2}\right) \right\vert \int\limits_{0}^{t}\tau
E\left( v_{m}\left( \tau \right) -v_{l}\left( \tau \right) \right) d\tau
\end{equation*}%
\begin{equation*}
+\frac{\widehat{c}_{1}}{\varepsilon }\int\limits_{0}^{t}\tau E\left(
v_{m}\left( \tau \right) -v_{l}\left( \tau \right) \right) \left\Vert
v_{mt}\left( \tau \right) \right\Vert _{L^{2}\left(
\mathbb{R}
^{n}\right) }d\tau
\end{equation*}%
\begin{equation*}
\leq (f\left( \varepsilon \right) +\widehat{c}_{1})\int\limits_{0}^{t}\left%
\Vert \nabla v_{m}(\tau )-\nabla v_{l}(\tau )\right\Vert _{L^{2}\left(
\mathbb{R}
^{n}\right) }^{2}d\tau
\end{equation*}%
\begin{equation*}
+\left( 2\max_{s_{1},s_{2}\epsilon \left[ 0,\varepsilon \right] }\left\vert
f\left( s_{1}\right) -f\left( s_{2}\right) \right\vert +\widehat{c}%
_{1}\varepsilon \right) \int\limits_{0}^{t}\tau E\left( v_{m}\left( \tau
\right) -v_{l}\left( \tau \right) \right) d\tau
\end{equation*}%
\begin{equation}
+\frac{\widehat{c}_{1}}{\varepsilon ^{3}}\int\limits_{0}^{t}\tau E\left(
v_{m}\left( \tau \right) -v_{l}\left( \tau \right) \right) \left\Vert
v_{mt}\left( \tau \right) \right\Vert _{L^{2}\left(
\mathbb{R}
^{n}\right) }^{2}d\tau \text{, \ }\forall t\in \lbrack 0,T]\text{.}
\tag{2.6}
\end{equation}%
Thus, setting $\gamma =2\max\limits_{s_{1},s_{2}\epsilon \left[
0,\varepsilon \right] }\left\vert f\left( s_{1}\right) -f\left( s_{2}\right)
\right\vert +\widehat{c}_{1}\varepsilon $ and $c_{\gamma }=\max \left\{
\frac{\widehat{c}_{1}}{\varepsilon ^{3}},\text{ }f\left( \varepsilon \right)
+\widehat{c}_{1}\right\} $, by (2.1), (2.2) and (2.6), we get the claim of
the lemma.
\end{proof}

Now, let us prove the asymptotic compactness of $\left\{ S(t)\right\}
_{t\geq 0}$ in $H^{2}\left(
\mathbb{R}
^{n}\right) \times L^{2}\left(
\mathbb{R}
^{n}\right) $.

\begin{lemma}
Assume that the conditions (1.3)-(1.4) hold and $B$ is a bounded subset of$\
H^{2}\left(
\mathbb{R}
^{n}\right) \times L^{2}\left(
\mathbb{R}
^{n}\right) $. Then for every sequence of the form $\left\{ S(t_{k})\varphi
_{k}\right\} _{k=1}^{\infty },$ where $\left\{ \varphi _{k}\right\}
_{k=1}^{\infty }\subset B$, $t_{k}\rightarrow \infty ,$ \ has a convergent
subsequence in $H^{2}\left(
\mathbb{R}
^{n}\right) \times L^{2}\left(
\mathbb{R}
^{n}\right) $.
\end{lemma}

\begin{proof}
Since $\left\{ \varphi _{k}\right\} _{k=1}^{\infty }$ is bounded in $%
H^{2}\left(
\mathbb{R}
^{n}\right) \times L^{2}\left(
\mathbb{R}
^{n}\right) $, by (1.3), (1.4) and (1.6) it follows that the sequence $%
\left\{ S\left( .\right) \varphi _{k}\right\} _{k=1}^{\infty }$ is bounded
in $C_{b}\left( 0,\infty ;H^{2}\left(
\mathbb{R}
^{n}\right) \times L^{2}\left(
\mathbb{R}
^{n}\right) \right) $, where $C_{b}\left( 0,\infty ;H^{2}\left(
\mathbb{R}
^{n}\right) \times L^{2}\left(
\mathbb{R}
^{n}\right) \right) $ is the space of continuously bounded functions from $%
\left[ 0,\infty \right) $ to $H^{2}\left(
\mathbb{R}
^{n}\right) \times L^{2}\left(
\mathbb{R}
^{n}\right) $. Then for any $T\geq 0$ \ there exists a subsequence $\left\{
k_{m}\right\} _{m=1}^{\infty }$ such that $t_{k_{m}}\geq T$, and
\begin{equation}
\begin{array}{c}
\left\Vert \nabla v_{m}\left( t\right) \right\Vert _{L^{2}\left(
\mathbb{R}
^{n}\right) }^{2}\rightarrow q\left( t\right) \text{ weakly star in }%
W^{1,\infty }\left( 0,\infty \right) \text{,}%
\end{array}
\tag{2.7}
\end{equation}
for some $q\in W^{1,\infty }\left( 0,\infty \right) $, where $\left(
v_{m}(t\right) ,v_{mt}\left( t\right) )=S(t+t_{k_{m}}-T)\varphi _{k_{m}}$.%
\newline
Taking into account (1.4) and (1.6), we find%
\begin{equation}
\int\limits_{0}^{T}\left\Vert v_{mt}(t)\right\Vert _{L^{2}\left(
\mathbb{R}
^{n}\right) }^{2}dt\leq c_{1}\text{, \ \ \ }\forall T\geq 0\text{.}
\tag{2.8}
\end{equation}
By (1.1)$_{1},$ we have%
\begin{equation*}
v_{mtt}(t,x)-v_{ltt}(t,x)+\Delta ^{2}\left( v_{m}(t,x)-v_{l}(t,x)\right)
+\alpha (x)\left( v_{mt}(t,x)-v_{lt}(t,x)\right) +\lambda \left(
v_{m}(t,x)-v_{l}(t,x)\right)
\end{equation*}%
\begin{equation}
-f(\left\Vert \nabla v_{m}\left( t\right) \right\Vert _{L^{2}\left(
\mathbb{R}
^{n}\right) })\Delta v_{m}(t,x)+f(\left\Vert \nabla v_{l}\left( t\right)
\right\Vert _{L^{2}\left(
\mathbb{R}
^{n}\right) })\Delta v_{l}(t,x)=0\text{.}  \tag{2.9}
\end{equation}
Multiplying (2.9) by $(v_{m}-v_{l})$ and integrating over $\left( 0,T\right)
\times
\mathbb{R}
^{n}$, we get%
\begin{equation*}
\int\limits_{0}^{T}\left\Vert \Delta \left( v_{m}\left( t\right)
-v_{l}\left( t\right) \right) \right\Vert _{L^{2}\left(
\mathbb{R}
^{n}\right) }^{2}dt+\lambda \int\limits_{0}^{T}\left\Vert v_{m}\left(
t\right) -v_{l}\left( t\right) \right\Vert _{L^{2}\left(
\mathbb{R}
^{n}\right) }^{2}dt
\end{equation*}%
\begin{equation*}
+\int\limits_{0}^{T}f(\left\Vert \nabla v_{m}\left( t\right) \right\Vert
_{L^{2}\left(
\mathbb{R}
^{n}\right) })\left\Vert \nabla v_{m}\left( t\right) -\nabla v_{l}\left(
t\right) \right\Vert _{L^{2}\left(
\mathbb{R}
^{n}\right) }^{2}dt\leq c_{2}+c_{2}\int\limits_{0}^{T}\left\Vert
v_{mt}(t)-v_{lt}(t)\right\Vert _{L^{2}\left(
\mathbb{R}
^{n}\right) }^{2}dt
\end{equation*}%
\begin{equation*}
+\int\limits_{0}^{T}\left\vert f(\left\Vert \nabla v_{m}\left( t\right)
\right\Vert _{L^{2}\left(
\mathbb{R}
^{n}\right) })-f(\left\Vert \nabla v_{l}\left( t\right) \right\Vert
_{L^{2}\left(
\mathbb{R}
^{n}\right) })\right\vert \left\Vert \nabla v_{l}(t,x)\right\Vert
_{L^{2}\left(
\mathbb{R}
^{n}\right) }\left\Vert \nabla v_{m}\left( t\right) -\nabla v_{l}\left(
t\right) \right\Vert _{L^{2}\left(
\mathbb{R}
^{n}\right) }dt.
\end{equation*}%
Taking into account (1.4), (2.7) and (2.8) in the last inequality and
passing to the limit, we obtain%
\begin{equation}
\underset{m\rightarrow \infty }{\lim \sup }\text{ }\underset{l\rightarrow
\infty }{\lim \sup }\int\limits_{0}^{T}E\left( v_{m}\left( t\right)
-v_{l}\left( t\right) \right) dt\leq c_{3},\text{ \ \ \ }\forall T\geq 0.
\tag{2.10}
\end{equation}

Multiplying (2.9) by $t\left( v_{mt}-v_{lt}\right) $, integrating over $%
\left( 0,T\right) \times
\mathbb{R}
^{n}$ and using integration by parts, by (1.3), we find%
\begin{equation*}
T\text{ }E\left( v_{m}\left( T\right) -v_{l}\left( T\right) \right) +\alpha
_{0}\int\limits_{0}^{T}t\left\Vert v_{mt}\left( T\right) -v_{lt}\left(
T\right) \right\Vert _{L^{2}\left(
\mathbb{R}
^{n}\right) }^{2}dt\leq \int\limits_{0}^{T}E\left( v_{m}\left( t\right)
-v_{l}\left( t\right) \right) dt
\end{equation*}%
\begin{equation*}
+\int\limits_{0}^{T}\int\limits_{%
\mathbb{R}
^{n}}t\left( f(\left\Vert \nabla v_{m}\left( t\right) \right\Vert
_{L^{2}\left(
\mathbb{R}
^{n}\right) })\Delta v_{m}(t,x)-f(\left\Vert \nabla v_{l}\left( t\right)
\right\Vert _{L^{2}\left(
\mathbb{R}
^{n}\right) })\Delta v_{l}(t,x)\right) \left( v_{mt}\left( t,x\right)
-v_{lt}\left( t,x\right) \right) dxdt\text{,}
\end{equation*}%
which, together with Lemma 2.1, gives%
\begin{equation*}
T\text{ }E\left( v_{m}\left( T\right) -v_{l}\left( T\right) \right) +\alpha
_{0}\int\limits_{0}^{T}t\left\Vert v_{mt}\left( T\right) -v_{lt}\left(
T\right) \right\Vert _{L^{2}\left(
\mathbb{R}
^{n}\right) }^{2}dt
\end{equation*}%
\begin{equation*}
\leq \int\limits_{0}^{T}E\left( v_{m}\left( t\right) -v_{l}\left( t\right)
\right) dt+\gamma \int\limits_{0}^{T}tE\left( v_{m}\left( t\right)
-v_{l}\left( t\right) \right) dt+c_{\gamma }\int\limits_{0}^{T}E\left(
v_{m}\left( t\right) -v_{l}\left( t\right) \right) dt
\end{equation*}%
\begin{equation}
+c_{\gamma }\int\limits_{0}^{T}tE\left( v_{m}\left( t\right) -v_{l}\left(
t\right) \right) \left\Vert v_{mt}\left( t\right) \right\Vert _{L^{2}\left(
\mathbb{R}
^{n}\right) }^{2}dt+K^{m,l}(T)\text{, }  \tag{2.11}
\end{equation}%
for every $\gamma >0$. \ Multiplying (2.9) by $\varepsilon t\left(
v_{m}-v_{l}\right) $ and integrating over $\left( 0,T\right) \times
\mathbb{R}
^{n}$, we get%
\begin{equation*}
\varepsilon \int\limits_{0}^{T}t\left\Vert \Delta \left( v_{m}\left(
t\right) -v_{l}\left( t\right) \right) \right\Vert _{L^{2}\left(
\mathbb{R}
^{n}\right) }^{2}dt+\varepsilon \lambda \int\limits_{0}^{T}t\left\Vert
v_{m}\left( t\right) -v_{l}\left( t\right) \right\Vert _{L^{2}\left(
\mathbb{R}
^{n}\right) }^{2}dt
\end{equation*}%
\begin{equation*}
\leq \varepsilon c_{3}TE\left( v_{m}\left( T\right) -v_{l}\left( T\right)
\right) +\varepsilon \int\limits_{0}^{T}t\left\Vert v_{mt}\left( t\right)
-v_{lt}\left( t\right) \right\Vert _{L^{2}\left(
\mathbb{R}
^{n}\right) }^{2}dt
\end{equation*}%
\begin{equation}
+\varepsilon c_{3}\int\limits_{0}^{T}\left\Vert v_{m}\left( t\right)
-v_{l}\left( t\right) \right\Vert _{L^{2}\left(
\mathbb{R}
^{n}\right) }^{2}dt+\varepsilon \widetilde{K}^{m,l}\left( T\right) ,
\tag{2.12}
\end{equation}%
where%
\begin{equation*}
\widetilde{K}^{m,l}\left( t\right) :=\int\limits_{0}^{t}\tau \left(
f(\left\Vert \nabla v_{m}\left( \tau \right) \right\Vert _{L^{2}\left(
\mathbb{R}
^{n}\right) }-f(\left\Vert \nabla v_{l}\left( \tau \right) \right\Vert
_{L^{2}\left(
\mathbb{R}
^{n}\right) }\right) \left\Vert \nabla v_{l}\left( \tau \right) \right\Vert
_{L^{2}\left(
\mathbb{R}
^{n}\right) }\left\Vert v_{m}\left( \tau \right) -v_{l}\left( \tau \right)
\right\Vert _{L^{2}\left(
\mathbb{R}
^{n}\right) }d\tau \text{,}
\end{equation*}%
and by (1.4) and (2.7), it is easy to see that
\begin{equation*}
\underset{m\rightarrow \infty }{\lim \sup }\text{ }\underset{l\rightarrow
\infty }{\lim \sup \text{ }}\left\Vert \widetilde{K}^{m,l}\right\Vert
_{C[0,T]}=0.
\end{equation*}%
Adding (2.11) to (2.12) and choosing $\gamma $ and $\varepsilon $ small
enough, we obtain%
\begin{equation*}
T\text{ }E\left( v_{m}\left( T\right) -v_{l}\left( T\right) \right) \leq
c_{4}\int\limits_{0}^{T}tE\left( v_{m}\left( t\right) -v_{l}\left( t\right)
\right) \left\Vert v_{mt}\left( t\right) \right\Vert _{L^{2}\left(
\mathbb{R}
^{n}\right) }^{2}dt
\end{equation*}%
\begin{equation*}
+c_{4}\int\limits_{0}^{T}E\left( v_{m}\left( t\right) -v_{l}\left( t\right)
\right) dt+c_{4}\left\vert K^{m,l}(T)\right\vert +c_{4}\left\vert \widetilde{%
K}^{m,l}(T)\right\vert ,\text{ \ \ \ }\forall T\geq 0.
\end{equation*}%
Now, denoting $y_{m,l}(t):=tE\left( v_{m}\left( t\right) -v_{l}\left(
t\right) \right) $ and applying Gronwall inequality, we get%
\begin{equation*}
y_{m,l}(T)
\end{equation*}%
\begin{equation*}
\leq c_{4}\left( \int\limits_{0}^{T}E\left( v_{m}\left( t\right)
-v_{l}\left( t\right) \right) dt+\left\Vert K^{m,l}\right\Vert
_{C[0,T]}+\left\Vert \widetilde{K}^{m,l}\right\Vert _{C[0,T]}\right)
e^{\int\limits_{0}^{T}\left\Vert v_{mt}\left( t\right) \right\Vert
_{L^{2}\left(
\mathbb{R}
^{n}\right) }^{2}dt},
\end{equation*}%
which, together with (2.8), yields%
\begin{equation*}
T\text{ }E\left( v_{m}\left( T\right) -v_{l}\left( T\right) \right) \leq
c_{5}\left( \int\limits_{0}^{T}E\left( v_{m}\left( t\right) -v_{l}\left(
t\right) \right) dt+\left\Vert K^{m,l}\right\Vert _{C[0,T]}+\left\Vert
\widetilde{K}^{m,l}\right\Vert _{C[0,T]}\right) ,
\end{equation*}%
for every $T\geq 0$. Passing to the limit in the above inequality and taking
into account (2.10), we find

\begin{equation*}
\underset{m\rightarrow \infty }{\lim \sup \text{ }}\underset{l\rightarrow
\infty }{\lim \sup }\text{ }T\text{ }E\left( v_{m}\left( T\right)
-v_{l}\left( T\right) \right) \leq c_{6},\text{ \ \ }\forall T\geq 0,
\end{equation*}%
which gives
\begin{equation*}
\underset{m\rightarrow \infty }{\lim \sup }\text{ }\underset{l\rightarrow
\infty }{\lim \sup }\left\Vert S(t_{k_{m}})\varphi
_{k_{m}}-S(t_{k_{l}})\varphi _{k_{l}}\right\Vert _{H^{2}\left(
\mathbb{R}
^{n}\right) \times L^{2}\left(
\mathbb{R}
^{n}\right) }\leq \frac{c_{7}}{\sqrt{T}},\text{ \ \ \ }\forall T>0.
\end{equation*}%
Consequently, we have%
\begin{equation*}
\underset{l\rightarrow \infty }{\lim \inf }\text{ }\underset{m\rightarrow
\infty }{\lim \inf }\left\Vert S(t_{k})\varphi _{k}-S(t_{m})\varphi
_{m}\right\Vert _{H^{2}\left(
\mathbb{R}
^{n}\right) \times L^{2}\left(
\mathbb{R}
^{n}\right) }=0\text{.}
\end{equation*}%
Thus, by using the argument at the end of the proof of [20, Lemma 3.4], we
complete the proof of the lemma.
\end{proof}

Since, by (1.3) and (1.5), problem (1.1)-(1.2) admidts a strict Lyapunov
function%
\begin{equation*}
\Phi \left( u\left( t\right) \right) =E\left( u\left( t\right) \right) +%
\frac{1}{2}F\left( \left\Vert \nabla u\left( t\right) \right\Vert
_{L^{2}\left(
\mathbb{R}
^{n}\right) }^{2}\right) -\int\limits_{%
\mathbb{R}
^{n}}h\left( x\right) u\left( t,x\right) dx,
\end{equation*}%
applying [18, Corollary 7.5.7], we have the following theorem.

\begin{theorem}
Under conditions (1.3)-(1.6) the semigroup $\left\{ S\left( t\right)
\right\} _{t\geq 0}$ generated by the problem (1.1)-(1.2) possesses a global
attractor $\mathcal{A}$ in $H^{2}\left(
\mathbb{R}
^{n}\right) \times L^{2}\left(
\mathbb{R}
^{n}\right) $ and $\mathcal{A=M}^{u}\left( \mathcal{N}\right) $.
\end{theorem}

\section{Regularity of the global attractor}

To prove the regularity of the global attractor, we start with the following
lemma.

\begin{lemma}
Assume that the conditions (1.3) and (1.4) hold and $B$ is a bounded subset
in $H^{4}\left(
\mathbb{R}
^{n}\right) \times H^{2}\left(
\mathbb{R}
^{n}\right) $. Then there exists a constant $C>0$ such that%
\begin{equation*}
\sup_{\varphi \in B}\left\Vert S\left( t\right) \varphi \right\Vert
_{H^{4}\left(
\mathbb{R}
^{n}\right) \times H^{2}\left(
\mathbb{R}
^{n}\right) }\leq C
\end{equation*}%
for all $t\geq 0$.
\end{lemma}

\begin{proof}
Let $\left( u_{0},u_{1}\right) \in B$ and $S\left( t\right) \left(
u_{0},u_{1}\right) :=\left( u\left( t\right) ,u_{t}\left( t\right) \right) $%
. Then we have $u\in $ $C\left( [0,\infty );H^{4}\left(
\mathbb{R}
^{n}\right) \right) \cap $ $C^{1}\left( [0,\infty );H^{2}\left(
\mathbb{R}
^{n}\right) \right) \cap $ $C^{2}\left( [0,\infty );L^{2}\left(
\mathbb{R}
^{n}\right) \right) $. Defining
\begin{equation*}
v\left( t,x\right) :=\frac{u\left( t+\tau ,x\right) -u\left( t,x\right) }{%
\tau }\text{, }\tau >0,
\end{equation*}%
by (1.1), we get%
\begin{equation*}
v_{tt}(t,x)+\Delta ^{2}v(t,x)+\alpha (x)v_{t}(t,x)+\lambda v(t,x)-f\left(
\left\Vert \nabla u\left( t\right) \right\Vert _{L^{2}\left(
\mathbb{R}
^{n}\right) }\right) \Delta v\left( t,x\right)
\end{equation*}%
\begin{equation}
-\frac{f(\left\Vert \nabla u\left( t+\tau \right) \right\Vert _{L^{2}\left(
\mathbb{R}
^{n}\right) })-f\left( \left\Vert \nabla u\left( t\right) \right\Vert
_{L^{2}\left(
\mathbb{R}
^{n}\right) }\right) }{\tau }\Delta u(t+\tau ,x)=0\text{, \ \ }(t,x)\in
(0,\infty )\times
\mathbb{R}
^{n}\text{.}  \tag{3.1}
\end{equation}%
Multiplying (3.1) by $v_{t}$ and integrating over $%
\mathbb{R}
^{n}$, we find%
\begin{equation*}
\frac{d}{dt}E(v(t))+\int\limits_{%
\mathbb{R}
^{n}}\alpha \left( x\right) \left\vert v_{t}\left( t,x\right) \right\vert
^{2}dx+\frac{1}{2}f\left( \left\Vert \nabla u\left( t\right) \right\Vert
_{L^{2}\left(
\mathbb{R}
^{n}\right) }\right) \frac{d}{dt}\left( \left\Vert \nabla v\left( t\right)
\right\Vert _{L^{2}\left(
\mathbb{R}
^{n}\right) }^{2}\right)
\end{equation*}%
\begin{equation}
-\frac{f(\left\Vert \nabla u\left( t+\tau \right) \right\Vert _{L^{2}\left(
\mathbb{R}
^{n}\right) })-f\left( \left\Vert \nabla u\left( t\right) \right\Vert
_{L^{2}\left(
\mathbb{R}
^{n}\right) }\right) }{\tau }\int\limits_{%
\mathbb{R}
^{n}}\Delta u(t+\tau ,x)v_{t}\left( t,x\right) dx=0\text{.}  \tag{3.2}
\end{equation}%
Since
\begin{equation*}
\left\vert \frac{d}{dt}f\left( \left\Vert \nabla u\left( t\right)
\right\Vert _{L^{2}\left(
\mathbb{R}
^{n}\right) }\right) \right\vert =\left\vert \frac{f^{\prime }\left(
\left\Vert \nabla u\left( t\right) \right\Vert _{L^{2}\left(
\mathbb{R}
^{n}\right) }\right) }{\left\Vert \nabla u\left( t\right) \right\Vert
_{L^{2}\left(
\mathbb{R}
^{n}\right) }}\left\langle \nabla u\left( t\right) ,\nabla u_{t}\left(
t\right) \right\rangle _{L^{2}\left(
\mathbb{R}
^{n}\right) }\right\vert
\end{equation*}%
\begin{equation*}
\leq \left\vert f^{\prime }\left( \left\Vert \nabla u\left( t\right)
\right\Vert _{L^{2}\left(
\mathbb{R}
^{n}\right) }\right) \right\vert \left\Vert \nabla u_{t}\left( t\right)
\right\Vert _{L^{2}\left(
\mathbb{R}
^{n}\right) },\text{ a.e. in }\left( 0,\infty \right) \text{,}
\end{equation*}%
considering (1.3) and (1.6) in (3.2), we obtain%
\begin{equation*}
\frac{d}{dt}\left( E(v(t))+\frac{1}{2}f\left( \left\Vert \nabla u\left(
t\right) \right\Vert _{L^{2}\left(
\mathbb{R}
^{n}\right) }\right) \left\Vert \nabla v\left( t\right) \right\Vert
_{L^{2}\left(
\mathbb{R}
^{n}\right) }^{2}\right) +\alpha _{0}\left\Vert v_{t}\left( t\right)
\right\Vert _{L^{2}\left(
\mathbb{R}
^{n}\right) }^{2}
\end{equation*}%
\begin{equation*}
\leq c_{1}\left( \left\Vert \nabla u_{t}\left( t\right) \right\Vert
_{L^{2}\left(
\mathbb{R}
^{n}\right) }\left\Vert \nabla v\left( t\right) \right\Vert _{L^{2}\left(
\mathbb{R}
^{n}\right) }^{2}+\left\Vert \nabla v\left( t\right) \right\Vert
_{L^{2}\left(
\mathbb{R}
^{n}\right) }\int\limits_{%
\mathbb{R}
^{n}}\Delta u(t+\tau ,x)v_{t}\left( t,x\right) dx\right)
\end{equation*}%
\begin{equation*}
\leq c_{2}\left( \left\Vert \nabla u_{t}\left( t\right) \right\Vert
_{L^{2}\left(
\mathbb{R}
^{n}\right) }\left\Vert v\left( t\right) \right\Vert _{H^{2}\left(
\mathbb{R}
^{n}\right) }\left\Vert v\left( t\right) \right\Vert _{L^{2}\left(
\mathbb{R}
^{n}\right) }+\left\Vert v\left( t\right) \right\Vert _{H^{2}\left(
\mathbb{R}
^{n}\right) }^{\frac{1}{2}}\left\Vert v\left( t\right) \right\Vert
_{L^{2}\left(
\mathbb{R}
^{n}\right) }^{\frac{1}{2}}\left\Vert v_{t}\left( t\right) \right\Vert
_{L^{2}\left(
\mathbb{R}
^{n}\right) }\right) \text{.}
\end{equation*}%
Since, by (1.6),
\begin{equation*}
\left\Vert v\left( t\right) \right\Vert _{L^{2}\left(
\mathbb{R}
^{n}\right) }=\left\Vert \frac{u\left( t+\tau ,x\right) -u\left( t,x\right)
}{\tau }\right\Vert _{L^{2}\left(
\mathbb{R}
^{n}\right) }\leq \sup_{0\leq t<\infty }\left\Vert u_{t}\left( t\right)
\right\Vert _{L^{2}\left(
\mathbb{R}
^{n}\right) }<\widehat{C},
\end{equation*}%
by the previous inequality, we get%
\begin{equation*}
\frac{d}{dt}\left( E(v(t))+\frac{1}{2}f\left( \left\Vert \nabla u\left(
t\right) \right\Vert _{L^{2}\left(
\mathbb{R}
^{n}\right) }\right) \left\Vert \nabla v\left( t\right) \right\Vert
_{L^{2}\left(
\mathbb{R}
^{n}\right) }^{2}\right) +\alpha _{0}\left\Vert v_{t}\left( t\right)
\right\Vert _{L^{2}\left(
\mathbb{R}
^{n}\right) }^{2}
\end{equation*}%
\begin{equation}
\leq c_{3}\left( \left\Vert \nabla u_{t}\left( t\right) \right\Vert
_{L^{2}\left(
\mathbb{R}
^{n}\right) }\left\Vert v\left( t\right) \right\Vert _{H^{2}\left(
\mathbb{R}
^{n}\right) }+\left\Vert v\left( t\right) \right\Vert _{H^{2}\left(
\mathbb{R}
^{n}\right) }^{\frac{1}{2}}\left\Vert v_{t}\left( t\right) \right\Vert
_{L^{2}\left(
\mathbb{R}
^{n}\right) }\right) \text{.}  \tag{3.3}
\end{equation}%
Multiplying (3.1) by $\varepsilon v$ and integrating over $%
\mathbb{R}
^{n}$, we find%
\begin{equation*}
\varepsilon \frac{d}{dt}\left( \left\langle \nabla v\left( t\right) ,\nabla
v_{t}\left( t\right) \right\rangle _{L^{2}\left(
\mathbb{R}
^{n}\right) }+\frac{1}{2}\int\limits_{%
\mathbb{R}
^{n}}\alpha \left( x\right) v\left( t,x\right) ^{2}dx\right) +\varepsilon
\left\Vert \Delta v\left( t\right) \right\Vert _{L^{2}\left(
\mathbb{R}
^{n}\right) }^{2}
\end{equation*}%
\begin{equation}
+\varepsilon \lambda \left\Vert v\left( t\right) \right\Vert _{L^{2}\left(
\mathbb{R}
^{n}\right) }^{2}\leq \varepsilon \left\Vert v_{t}\left( t\right)
\right\Vert _{L^{2}\left(
\mathbb{R}
^{n}\right) }^{2}+\varepsilon c_{4}\left\Vert v\left( t\right) \right\Vert
_{H^{1}\left(
\mathbb{R}
^{n}\right) }.  \tag{3.4}
\end{equation}%
Considering (3.3) and (3.4), for sufficiently small $\varepsilon >0$ and
applying Young inequality, we obtain \
\begin{equation}
\frac{d}{dt}\Psi \left( t\right) +c_{5}E\left( v\left( t\right) \right) \leq
c_{6}+c_{6}\left\Vert \nabla u_{t}\left( t\right) \right\Vert _{L^{2}\left(
\mathbb{R}
^{n}\right) }^{2},  \tag{3.5}
\end{equation}%
where $c_{5}>0$ and
\begin{equation*}
\Psi \left( t\right) :=E\left( v\left( t\right) \right) +\frac{1}{2}f\left(
\left\Vert \nabla u\left( t\right) \right\Vert \right) \left\Vert \nabla
v\left( t\right) \right\Vert _{L^{2}\left(
\mathbb{R}
^{n}\right) }^{2}+\varepsilon \left\langle \nabla v\left( t\right) ,\nabla
v_{t}\left( t\right) \right\rangle _{L^{2}\left(
\mathbb{R}
^{n}\right) }+\frac{\varepsilon }{2}\int\limits_{%
\mathbb{R}
^{n}}\alpha \left( x\right) \left\vert v\left( t,x\right) \right\vert ^{2}dx.
\end{equation*}%
Since $\varepsilon >0$ is sufficiently small, there exist constants $c>0,$ $%
\widetilde{c}>0$ such that%
\begin{equation}
cE\left( v\left( t\right) \right) \leq \Psi \left( t\right) \leq \widetilde{c%
}E\left( v\left( t\right) \right) \text{.}  \tag{3.6}
\end{equation}%
Then, by (3.5) and (3.6), we have%
\begin{equation*}
\frac{d}{dt}\Psi \left( t\right) +c_{7}\Psi \left( t\right) \leq
c_{6}+c_{6}\left\Vert \nabla u_{t}\left( t\right) \right\Vert _{L^{2}\left(
\mathbb{R}
^{n}\right) }^{2}
\end{equation*}%
which yields%
\begin{equation*}
\Psi \left( t\right) \leq c_{8}+e^{-c_{7}t}c_{8}\int_{0}^{t}\left\Vert
\nabla u_{t}\left( s\right) \right\Vert _{L^{2}\left(
\mathbb{R}
^{n}\right) }^{2}e^{c_{7}s}ds
\end{equation*}%
\begin{equation*}
\leq c_{8}+c_{8}\sup_{t\in \left[ 0,T\right] }\left( \left\Vert u_{t}\left(
t\right) \right\Vert _{L^{2}\left(
\mathbb{R}
^{n}\right) }\left\Vert u_{t}\left( t\right) \right\Vert _{H^{2}\left(
\mathbb{R}
^{n}\right) }\right) ,
\end{equation*}%
for every $T\geq 0$. Taking into account (1.6) and (3.6) in the last
inequality, we find%
\begin{equation*}
E\left( v\left( t\right) \right) \leq c_{9}+c_{9}\sup_{t\in \left[ 0,T\right]
}\left\Vert u_{t}\left( t\right) \right\Vert _{H^{2}\left(
\mathbb{R}
^{n}\right) },\text{ \ }\forall t\in \lbrack 0,T],
\end{equation*}%
for every $T\geq 0$. Passing to limit as $\tau \rightarrow 0$ in the above
inequality, from the definition of $v$, we obtain%
\begin{equation*}
E\left( u_{t}\left( t\right) \right) \leq c_{9}+c_{9}\sup_{t\in \left[ 0,T%
\right] }\left\Vert u_{t}\left( t\right) \right\Vert _{H^{2}\left(
\mathbb{R}
^{n}\right) },\text{ \ \ }\forall t\in \lbrack 0,T],
\end{equation*}%
for every $T\geq 0$. Thus, after taking supremum on $\left[ 0,T\right] $ and
applying Young inequality, we have%
\begin{equation*}
E\left( u_{t}\left( t\right) \right) \leq c_{10},\text{ \ \ }\forall t\geq 0.
\end{equation*}%
Taking into account this estimate in (1.1), we find that%
\begin{equation*}
\left\Vert u\left( t\right) \right\Vert _{H^{4}\left(
\mathbb{R}
^{n}\right) }\leq c_{11}
\end{equation*}%
which, together with previous inequality, yields%
\begin{equation*}
\left\Vert \left( u\left( t\right) ,u_{t}\left( t\right) \right) \right\Vert
_{H^{4}\left(
\mathbb{R}
^{n}\right) \times H^{2}\left(
\mathbb{R}
^{n}\right) }\leq C,
\end{equation*}%
for some constant $C>0$.
\end{proof}

Now we can show the regularity of the attractor.

\begin{theorem}
Under the assumptions of Theorem 1.1, the global attractor $\mathcal{A}$ for
the problem (1.1)-(1.2) is bounded in $H^{4}\left(
\mathbb{R}
^{n}\right) \times H^{2}\left(
\mathbb{R}
^{n}\right) $.
\end{theorem}

\begin{proof}
Let $\theta \in \mathcal{A}$. By the invariance of $\mathcal{A}$, it follows
that (see [21, p. 159]) there exists an invariant trajectory $\gamma
=\left\{ U\left( t\right) =\left( u\left( t\right) ,u_{t}\left( t\right)
\right) :t\in
\mathbb{R}
\right\} \subset \mathcal{A}$ such that $U\left( 0\right) =\theta $. By an
invariant trajectory \ we mean a curve $\gamma =\left\{ U\left( t\right)
:t\in
\mathbb{R}
\right\} $ such that $S\left( t\right) U\left( \tau \right) =U\left( t+\tau
\right) $ for all $t\geq 0$ and $\tau \in
\mathbb{R}
$ (see [21, p. 157])

In the case when $h\equiv 0$ in equation (1.1), by (1.4), it follows that
the stationary point set $\mathcal{N}=\left\{ \left( 0,0\right) \right\} .$
By Theorem 1.1 and the definition of unstable manifold, we have
\begin{equation}
\text{\ \ }\lim_{t\rightarrow -\infty }\inf_{w\in \mathcal{N}}\left\Vert
U\left( t\right) -w\right\Vert _{H^{2}(%
\mathbb{R}
^{n})\times L^{2}(%
\mathbb{R}
^{n})}=0.  \tag{3.7}
\end{equation}%
Then, from the monotonicity of the Lyapunov function $\Phi \left( \cdot
\right) $, we have $\mathcal{A}=\left\{ \left( 0,0\right) \right\} $.

So, we will consider the case $h\neq 0$. In this case, it is clear that $%
\mathcal{N}$ does not contain $\left( 0,0\right) $. Since $\mathcal{N}$ is
compact (because it is a closed subset of $\mathcal{A}$), by (3.7), we
obtain that there exists $t_{0}\in \left( -\infty ,0\right) $ such that
\begin{equation}
\left\Vert U\left( t\right) \right\Vert _{H^{2}(%
\mathbb{R}
^{n})\times L^{2}(%
\mathbb{R}
^{n})}>c_{0},\text{ \ }\forall t\leq t_{0},  \tag{3.8}
\end{equation}%
for some $c_{0}>0$ which depends on $\mathcal{N}$ and is independent of $%
U\left( t\right) $. Now, again defining%
\begin{equation*}
v\left( t,x\right) :=\frac{u\left( t+\tau ,x\right) -u\left( t,x\right) }{%
\tau },
\end{equation*}%
we have equation (3.1). Multiplying (3.1) by $v_{t}$ and integrating over $%
\mathbb{R}
^{n}$, by using (3.8), we get%
\begin{equation*}
\frac{d}{dt}\left( E(v\left( t\right) )+\frac{1}{2}f\left( \left\Vert \nabla
u\left( t\right) \right\Vert \right) \left\Vert \nabla v\left( t\right)
\right\Vert _{L^{2}\left(
\mathbb{R}
^{n}\right) }^{2}\right) +\alpha _{0}\left\Vert v_{t}\left( t\right)
\right\Vert _{L^{2}\left(
\mathbb{R}
^{n}\right) }^{2}
\end{equation*}%
\begin{equation*}
\leq \frac{f^{\prime }\left( \left\Vert \nabla u\left( t\right) \right\Vert
\right) }{\left\Vert \nabla u\left( t\right) \right\Vert }\left\Vert \Delta
u\left( t\right) \right\Vert \left\Vert u_{t}\left( t\right) \right\Vert
\left\Vert \nabla v\left( t\right) \right\Vert _{L^{2}\left(
\mathbb{R}
^{n}\right) }^{2}
\end{equation*}%
\begin{equation*}
+\widehat{c}_{1}\left\Vert \nabla v\left( t\right) \right\Vert _{L^{2}\left(
\mathbb{R}
^{n}\right) }\int\limits_{%
\mathbb{R}
^{n}}\Delta u(t+\tau ,x)v_{t}\left( t,x\right) dx
\end{equation*}%
\begin{equation*}
\leq \widehat{c}_{2}\left\Vert \nabla v\left( t\right) \right\Vert
_{L^{2}\left(
\mathbb{R}
^{n}\right) }^{2}+\widehat{c}_{2}\left\Vert \nabla v\left( t\right)
\right\Vert _{L^{2}\left(
\mathbb{R}
^{n}\right) }\left\Vert v_{t}\left( t\right) \right\Vert _{L^{2}\left(
\mathbb{R}
^{n}\right) }
\end{equation*}%
\begin{equation}
\leq \widehat{c}_{3}\left\Vert v\left( t\right) \right\Vert _{H^{2}\left(
\mathbb{R}
^{n}\right) }+\widehat{c}_{3}\left\Vert v\left( t\right) \right\Vert
_{H^{2}\left(
\mathbb{R}
^{n}\right) }^{\frac{1}{2}}\left\Vert v_{t}\left( t\right) \right\Vert
_{L^{2}\left(
\mathbb{R}
^{n}\right) }.  \tag{3.9}
\end{equation}%
By using (3.4), (3.6) and (3.9), we find \
\begin{equation*}
\frac{d}{dt}\Psi \left( t\right) +\widehat{c}_{4}\Psi \left( t\right) \leq
\widehat{c}_{5}\text{, \ \ \ }\forall t\leq t_{0},
\end{equation*}%
which yields%
\begin{equation*}
\Psi \left( t\right) \leq \widehat{c}_{6}+e^{\widehat{c}_{4}\left(
s-t\right) }\Psi \left( s\right) ,\text{ \ \ \ \ }s\leq t\leq t_{0},
\end{equation*}%
where $\widehat{c}_{4}>0$. Then, passing to limit as $s\rightarrow -\infty $
and taking into account that $\cup _{t\in
\mathbb{R}
}U\left( t\right) \subset \mathcal{A}$, we have%
\begin{equation*}
\Psi \left( t\right) \leq \widehat{c}_{6},
\end{equation*}%
which, by (3.6), gives%
\begin{equation*}
E\left( v\left( t\right) \right) \leq \widehat{c}_{7}.
\end{equation*}%
Now, passing to limit as $\tau \rightarrow 0$ in the last inequality, we
obtain%
\begin{equation}
E\left( u_{t}\left( t\right) \right) \leq \widehat{c}_{7},\text{ \ \ \ }%
\forall t\leq t_{0}.  \tag{3.10}
\end{equation}%
Considering (3.10) in (1.1), we find%
\begin{equation*}
\left\Vert u\left( t\right) \right\Vert _{H^{4}\left(
\mathbb{R}
^{n}\right) }\leq \widehat{c}_{8},\text{ \ \ \ }\forall t\leq t_{0},
\end{equation*}%
which, together with (3.10), yields%
\begin{equation*}
\left\Vert \left( u\left( t\right) ,u_{t}\left( t\right) \right) \right\Vert
_{H^{4}\left(
\mathbb{R}
^{n}\right) \times H^{2}\left(
\mathbb{R}
^{n}\right) }\leq \widehat{c}_{9},\text{ \ \ \ }\forall t\leq t_{0}.
\end{equation*}%
Thus, applying Lemma 3.1 to the set $B=\left\{ \left( u\left( t\right)
,u_{t}\left( t\right) \right) :t\in \left( -\infty ,t_{0}\right] \right\} $,
we obtain%
\begin{equation*}
\left\Vert \theta \right\Vert _{H^{4}\left(
\mathbb{R}
^{n}\right) \times H^{2}\left(
\mathbb{R}
^{n}\right) }\leq C,
\end{equation*}%
where $C>0$ is a constant independent of $\theta $.
\end{proof}

\section{Finite dimensionality of the global attractor}

In this section, we will use the idea of the [12] to obtain the finite
dimensionality. Let us start with the following lemma.

\begin{lemma}
Assume that the conditions (1.3) and (1.4) hold and $u\in W^{1,\infty
}\left( 0,\infty ;H^{2}\left(
\mathbb{R}
^{n}\right) \right) $ such that%
\begin{equation}
\left\Vert u\right\Vert _{W^{1,\infty }\left( 0,\infty ;H^{2}\left(
\mathbb{R}
^{n}\right) \right) }+\int_{0}^{\infty }\left\Vert u_{t}(t)\right\Vert
_{L^{2}\left(
\mathbb{R}
^{n}\right) }^{2}dt<c\text{.}  \tag{4.1}
\end{equation}%
for some constant $c>0$. Also, let $\left\{ T\left( t,\tau \right) \right\}
_{t\geq \tau }$ be the process generated by the problem%
\begin{equation}
\left\{
\begin{array}{c}
v_{tt}+\Delta ^{2}v+\alpha (x)v_{t}+\lambda v-f\left( \left\Vert \nabla
u\left( t\right) \right\Vert _{L^{2}\left(
\mathbb{R}
^{n}\right) }\right) \Delta v=0\text{, \ \ }t\geq \tau \text{, } \\
v\left( \tau \right) =v_{0}\text{, \ }v_{t}\left( \tau \right) =v_{1},\text{
\ \ \ }\tau \geq 0\text{ \ \ }%
\end{array}%
\right.  \tag{4.2}
\end{equation}%
in $H^{2}\left(
\mathbb{R}
^{n}\right) \times L^{2}\left(
\mathbb{R}
^{n}\right) $. Then there exist $M=M\left( c\right) >1$ and $\omega =\omega
\left( c\right) >0$ such that%
\begin{equation*}
\left\Vert T\left( t,\tau \right) \right\Vert _{L\left( H^{2\left(
1+i\right) }\left(
\mathbb{R}
^{n}\right) \times H^{2i}\left(
\mathbb{R}
^{n}\right) \right) }\leq Me^{-\omega \left( t-\tau \right) }\text{, \ \ \ }%
\forall t\geq \tau
\end{equation*}%
where $i=0,1$ and $L\left( X\right) $ is the space of linear bounded
operators in $X.$
\end{lemma}

\begin{proof}
By using the multiplier $(v_{t}+\varepsilon v)$ as in Lemma 3.1, for
sufficiently small $\varepsilon >0$ and applying Young inequality, we get \
\begin{equation*}
\frac{d}{dt}\Psi \left( t\right) +\gamma E\left( v\left( t\right) \right)
\leq c_{1}\left\Vert \nabla u_{t}\left( t\right) \right\Vert _{L^{2}\left(
\mathbb{R}
^{n}\right) }\left\Vert \nabla v\left( t\right) \right\Vert _{L^{2}\left(
\mathbb{R}
^{n}\right) }^{2},
\end{equation*}%
for some $\gamma >0$. Then, by using interpolation, (3.6) and (4.1), we find%
\begin{equation*}
\frac{d}{dt}\Psi \left( t\right) +\gamma \Psi \left( v\left( t\right)
\right) \leq c_{2}\left\Vert u_{t}\left( t\right) \right\Vert _{L^{2}\left(
\mathbb{R}
^{n}\right) }^{\frac{1}{2}}\Psi \left( t\right) .
\end{equation*}%
Hence, by Gronwall inequality and (3.6), we have,
\begin{equation}
E\left( t\right) \leq c_{3}E\left( \tau \right) e^{c_{2}\int\limits_{\tau
}^{t}\left\Vert u_{t}\left( \sigma \right) \right\Vert _{L^{2}\left(
\mathbb{R}
^{n}\right) }^{\frac{1}{2}}d\sigma -\gamma \left( t-\tau \right) }\text{.}
\tag{4.3}
\end{equation}%
Since, by Holder inequality and (4.1),%
\begin{equation*}
\int\limits_{\tau }^{t}\left\Vert u_{t}\left( \sigma \right) \right\Vert
_{L^{2}\left(
\mathbb{R}
^{n}\right) }^{\frac{1}{2}}d\sigma \leq c_{4}\left( t-\tau \right) ^{\frac{3%
}{4}},
\end{equation*}%
from (4.3) it follows that
\begin{equation}
\left\Vert T\left( t,\tau \right) \right\Vert _{L\left( H^{2}\left(
\mathbb{R}
^{n}\right) \times L^{2}\left(
\mathbb{R}
^{n}\right) \right) }\leq M_{1}e^{-\omega \left( t-\tau \right) },\text{ \ \
\ }\forall t\geq \tau ,  \tag{4.4}
\end{equation}%
for some $M_{1}>1$ and $\omega >0.$

Now, let us define $w:=v_{t}$. Then $w$ is the solution of the following
equation%
\begin{equation*}
w_{tt}(t,x)+\Delta ^{2}w(t,x)+\alpha (x)w_{t}(t,x)+\lambda w(t,x)-f\left(
\left\Vert \nabla u\left( t\right) \right\Vert _{L^{2}\left(
\mathbb{R}
^{n}\right) }\right) \Delta w\left( t,x\right)
\end{equation*}%
\begin{equation*}
-\frac{f^{\prime }\left( \left\Vert \nabla u\left( t\right) \right\Vert
_{L^{2}\left(
\mathbb{R}
^{n}\right) }\right) }{\left\Vert \nabla u\left( t\right) \right\Vert
_{L^{2}\left(
\mathbb{R}
^{n}\right) }}\left\langle \nabla u_{t}\left( t\right) ,\nabla u\left(
t\right) \right\rangle _{L^{2}\left(
\mathbb{R}
^{n}\right) }\Delta v\left( t,x\right) =0\text{, \ \ }t\geq \tau \text{, }%
x\in
\mathbb{R}
^{n},
\end{equation*}%
and by the variation of parameters formula, we have%
\begin{equation*}
W\left( t\right) =T\left( t,\tau \right) W\left( \tau \right)
+\int\limits_{\tau }^{t}T\left( t,s\right) G\left( s\right) ds,
\end{equation*}%
where $W\left( t\right) :=\left( w\left( t\right) ,w_{t}\left( t\right)
\right) $ and $G\left( t\right) :=\left( 0,\frac{f^{\prime }\left(
\left\Vert \nabla u\left( t\right) \right\Vert _{L^{2}\left(
\mathbb{R}
^{n}\right) }\right) }{\left\Vert \nabla u\left( t\right) \right\Vert
_{L^{2}\left(
\mathbb{R}
^{n}\right) }}\left\langle \nabla u_{t}\left( t\right) ,\nabla u\left(
t\right) \right\rangle _{L^{2}\left(
\mathbb{R}
^{n}\right) }\Delta v\left( t\right) \right) $. Therefore, by using (4.4),
we find%
\begin{equation*}
\left\Vert W\left( t\right) \right\Vert _{H^{2}\left(
\mathbb{R}
^{n}\right) \times L^{2}\left(
\mathbb{R}
^{n}\right) }\leq \left\Vert T\left( t,\tau \right) W\left( \tau \right)
\right\Vert _{H^{2}\left(
\mathbb{R}
^{n}\right) \times L^{2}\left(
\mathbb{R}
^{n}\right) }+\int\limits_{\tau }^{t}\left\Vert T\left( t,s\right) G\left(
s\right) \right\Vert _{H^{2}\left(
\mathbb{R}
^{n}\right) \times L^{2}\left(
\mathbb{R}
^{n}\right) }ds
\end{equation*}%
\begin{equation*}
\leq M_{1}e^{-\omega \left( t-\tau \right) }\left\Vert W\left( \tau \right)
\right\Vert _{H^{2}\left(
\mathbb{R}
^{n}\right) \times L^{2}\left(
\mathbb{R}
^{n}\right) }+c_{5}\int\limits_{\tau }^{t}e^{-\omega \left( t-s\right)
}\left\Vert G\left( s\right) \right\Vert _{H^{2}\left(
\mathbb{R}
^{n}\right) \times L^{2}\left(
\mathbb{R}
^{n}\right) }ds
\end{equation*}%
\begin{equation*}
\leq M_{1}e^{-\omega \left( t-\tau \right) }\left\Vert W\left( \tau \right)
\right\Vert _{H^{2}\left(
\mathbb{R}
^{n}\right) \times L^{2}\left(
\mathbb{R}
^{n}\right) }+c_{6}\int\limits_{\tau }^{t}e^{-\omega \left( t-s\right)
}\left\Vert v\left( s\right) \right\Vert _{H^{2}\left(
\mathbb{R}
^{n}\right) }ds
\end{equation*}%
\begin{equation*}
\leq M_{1}e^{-\omega \left( t-\tau \right) }\left\Vert W\left( \tau \right)
\right\Vert _{H^{2}\left(
\mathbb{R}
^{n}\right) \times L^{2}\left(
\mathbb{R}
^{n}\right) }+c_{7}\int\limits_{\tau }^{t}e^{-\omega \left( t-s\right)
}e^{-\omega \left( s-\tau \right) }\left\Vert (v\left( \tau \right)
,v_{t}\left( \tau \right) )\right\Vert _{H^{2}\left(
\mathbb{R}
^{n}\right) \times L^{2}\left(
\mathbb{R}
^{n}\right) }ds
\end{equation*}%
\begin{equation*}
\leq c_{8}e^{-\omega \left( t-\tau \right) }\left( \left\Vert W\left( \tau
\right) \right\Vert _{H^{2}\left(
\mathbb{R}
^{n}\right) \times L^{2}\left(
\mathbb{R}
^{n}\right) }+\left\Vert (v\left( \tau \right) ,v_{t}\left( \tau \right)
)\right\Vert _{H^{2}\left(
\mathbb{R}
^{n}\right) \times L^{2}\left(
\mathbb{R}
^{n}\right) }\right) ,\text{ \ }\forall t\geq \tau .
\end{equation*}%
Thus, the last inequality, together with (4.2)$_{1}$, gives
\begin{equation*}
\left\Vert T\left( t,\tau \right) \right\Vert _{L\left( H^{4}\left(
\mathbb{R}
^{n}\right) \times H^{2}\left(
\mathbb{R}
^{n}\right) \right) }\leq M_{2}e^{-\omega \left( t-\tau \right) }\text{, \ \
\ \ }\forall t\geq \tau ,
\end{equation*}%
for some $M_{2}>1$.
\end{proof}

Now, we can give the theorem about the finite dimensionality of the global
attractor.

\begin{theorem}
The fractal dimension of the global attractor $\mathcal{A}$ is finite.
\end{theorem}

\begin{proof}
Let $\theta _{1},$ $\theta _{2}\in \mathcal{A}$ and $\left( u\left( t\right)
,u_{t}\left( t\right) \right) =S\left( t\right) \theta _{1},$ $\left(
v\left( t\right) ,v_{t}\left( t\right) \right) =S\left( t\right) \theta _{2}$%
. Define $w\left( t\right) :=v\left( t\right) -u\left( t\right) $. Then, we
find
\begin{equation*}
w_{tt}(t,x)+\Delta ^{2}w(t,x)+\alpha (x)w_{t}(t,x)+\lambda w(t,x)-f\left(
\left\Vert \nabla u\left( t\right) \right\Vert _{L^{2}\left(
\mathbb{R}
^{n}\right) }\right) \Delta w\left( t,x\right)
\end{equation*}%
\begin{equation}
-\left( f\left( \left\Vert \nabla v\left( t\right) \right\Vert _{L^{2}\left(
\mathbb{R}
^{n}\right) }\right) -f\left( \left\Vert \nabla u\left( t\right) \right\Vert
_{L^{2}\left(
\mathbb{R}
^{n}\right) }\right) \right) \Delta v\left( t,x\right) =0\text{.}  \tag{4.5}
\end{equation}%
\ Hence, by the variation of parameters formula, we have
\begin{equation*}
\left( w\left( t\right) ,w_{t}\left( t\right) \right) =T\left( t,0\right)
\left( w\left( 0\right) ,w_{t}\left( 0\right) \right)
+\int\limits_{0}^{t}T\left( t,\tau \right) \widehat{G}\left( \tau \right)
d\tau ,
\end{equation*}%
where $\widehat{G}\left( t\right) =\left( 0,\left( f\left( \left\Vert \nabla
v\left( t\right) \right\Vert _{L^{2}\left(
\mathbb{R}
^{n}\right) }\right) -f\left( \left\Vert \nabla u\left( t\right) \right\Vert
_{L^{2}\left(
\mathbb{R}
^{n}\right) }\right) \right) \Delta v\left( t\right) \right) $. By Lemma
4.1, we get%
\begin{equation*}
\left\Vert S\left( t\right) \theta _{2}-S\left( t\right) \theta
_{1}\right\Vert _{H^{4}\left(
\mathbb{R}
^{n}\right) \times H^{2}\left(
\mathbb{R}
^{n}\right) }
\end{equation*}%
\begin{equation}
\leq Me^{-\omega t}\left\Vert \theta _{2}-\theta _{1}\right\Vert
_{H^{4}\left(
\mathbb{R}
^{n}\right) \times H^{2}\left(
\mathbb{R}
^{n}\right) }+\widetilde{c}_{1}\int\limits_{0}^{t}e^{-\omega \left( t-\tau
\right) }\left\Vert S\left( \tau \right) \theta _{2}-S\left( \tau \right)
\theta _{1}\right\Vert _{H^{2}\left(
\mathbb{R}
^{n}\right) \times L^{2}(%
\mathbb{R}
^{n})}d\tau .\text{ }  \tag{4.6}
\end{equation}%
Applying Gronwall lemma to (4.6), we obtain%
\begin{equation}
\left\Vert S\left( t\right) \theta _{2}-S\left( t\right) \theta
_{1}\right\Vert _{H^{4}\left(
\mathbb{R}
^{n}\right) \times H^{2}\left(
\mathbb{R}
^{n}\right) }\leq Me^{\left( \widetilde{c}_{2}-\omega \right) t}\left\Vert
\theta _{2}-\theta _{1}\right\Vert _{H^{4}\left(
\mathbb{R}
^{n}\right) \times H^{2}\left(
\mathbb{R}
^{n}\right) }\text{, \ \ }\forall t\geq 0\text{.}  \tag{4.7}
\end{equation}%
Also, by (4.6), we have%
\begin{equation*}
\left\Vert S\left( t\right) \theta _{2}-S\left( t\right) \theta
_{1}\right\Vert _{H^{4}\left(
\mathbb{R}
^{n}\right) \times H^{2}\left(
\mathbb{R}
^{n}\right) }
\end{equation*}%
\begin{equation*}
\leq Me^{-\omega t}\left\Vert \theta _{2}-\theta _{1}\right\Vert
_{H^{4}\left(
\mathbb{R}
^{n}\right) \times H^{2}\left(
\mathbb{R}
^{n}\right) }+\frac{\widetilde{c}_{1}}{\omega }\sup_{0\leq \tau \leq
t}\left\Vert S\left( \tau \right) \theta _{2}-S\left( \tau \right) \theta
_{1}\right\Vert _{H^{2}\left( B\left( 0,r\right) \right) \times L^{2}\left(
B\left( 0,r\right) \right) }
\end{equation*}%
\begin{equation}
+\widetilde{c}_{1}\int\limits_{0}^{t}e^{-\omega \left( t-\tau \right)
}\left\Vert S\left( \tau \right) \theta _{2}-S\left( \tau \right) \theta
_{1}\right\Vert _{H^{2}\left(
\mathbb{R}
^{n}\backslash B\left( 0,r\right) \right) \times L^{2}(%
\mathbb{R}
^{n}\backslash B\left( 0,r\right) )}d\tau \text{, \ \ }\forall t\geq 0\text{%
\ and }\forall r>0\text{,}  \tag{4.8}
\end{equation}%
where $B\left( 0,r\right) =\left\{ x:x\in
\mathbb{R}
^{n},\left\vert x\right\vert <r\right\} $.

Now, we will estimate the integral term on the right hand side of (4.8). Let
$\eta \in C^{\infty }\left(
\mathbb{R}
^{n}\right) $, $0\leq \eta \left( x\right) \leq 1$, $\eta \left( x\right)
=\left\{
\begin{array}{c}
0,\text{ }\left\vert x\right\vert \leq 1 \\
1,\text{ }\left\vert x\right\vert \geq 2%
\end{array}%
\right. $ and $\eta _{r}\left( x\right) =\eta \left( \frac{x}{r}\right) $.
Multiplying (4.5) by $\eta _{r}$ and denoting $w_{r}\left( t\right) =\eta
_{r}w\left( t\right) $, we get%
\begin{equation*}
w_{rtt}(t,x)+\Delta ^{2}w_{r}(t,x)+\alpha (x)w_{rt}(t,x)+\lambda
w_{r}(t,x)-f\left( \left\Vert \nabla u\left( t\right) \right\Vert
_{L^{2}\left(
\mathbb{R}
^{n}\right) }\right) \Delta w_{r}\left( t,x\right)
\end{equation*}%
\begin{equation*}
-\eta _{r}\left( f\left( \left\Vert \nabla v\left( t\right) \right\Vert
_{L^{2}\left(
\mathbb{R}
^{n}\right) }\right) -f\left( \left\Vert \nabla u\left( t\right) \right\Vert
_{L^{2}\left(
\mathbb{R}
^{n}\right) }\right) \right) \Delta v\left( t,x\right) =f_{1}\left( t\right)
\text{, }
\end{equation*}%
where%
\begin{equation*}
f_{1}\left( t\right) =\Delta ^{2}\eta _{r}w+2\Delta \eta _{r}\Delta
w+2\sum_{i=1}^{n}\left( \Delta \eta _{r}\right)
_{x_{i}}w_{x_{i}}+2\sum_{i=1}^{n}\left( \eta _{r}\right) _{x_{i}}\Delta
w_{x_{i}}+4\sum_{i,j=1}^{n}\left( \eta _{r}\right)
_{x_{i}x_{j}}w_{x_{i}x_{j}}
\end{equation*}%
\begin{equation*}
-\Delta \eta _{r}f\left( \left\Vert \nabla u\left( t\right) \right\Vert
_{L^{2}\left(
\mathbb{R}
^{n}\right) }\right) w-2f\left( \left\Vert \nabla u\left( t\right)
\right\Vert _{L^{2}\left(
\mathbb{R}
^{n}\right) }\right) \sum_{i=1}^{n}\left( \eta _{r}\right) _{x_{i}}w_{x_{i}}%
\text{.}
\end{equation*}%
Then, by the variation of parameters formula, we have%
\begin{equation}
\left( w_{r}\left( t\right) ,w_{rt}\left( t\right) \right) =T\left(
t,0\right) \left( w_{r}\left( 0\right) ,w_{rt}\left( 0\right) \right)
+\int\limits_{0}^{t}T\left( t,\tau \right) G_{r}\left( \tau \right) d\tau
\text{,}  \tag{4.9}
\end{equation}%
where%
\begin{equation*}
G_{r}\left( t\right) :=\left( 0,\eta _{r}\left( f\left( \left\Vert \nabla
v\left( t\right) \right\Vert _{L^{2}\left(
\mathbb{R}
^{n}\right) }\right) -f\left( \left\Vert \nabla u\left( t\right) \right\Vert
_{L^{2}\left(
\mathbb{R}
^{n}\right) }\right) \right) \Delta v\left( t\right) +f_{1}\left( t\right)
\right) \text{.}
\end{equation*}%
Hence, applying Lemma 4.1 to (4.9) and taking into account (4.7), we obtain%
\begin{equation*}
\left\Vert \left( w_{r}\left( t\right) ,w_{rt}\left( t\right) \right)
\right\Vert _{H^{2}\left(
\mathbb{R}
^{n}\right) \times L^{2}\left(
\mathbb{R}
^{n}\right) }\leq Me^{-\omega t}\left\Vert \theta _{2}-\theta
_{1}\right\Vert _{H^{2}\left(
\mathbb{R}
^{n}\right) \times L^{2}\left(
\mathbb{R}
^{n}\right) }
\end{equation*}%
\begin{equation*}
+\widetilde{c}_{3}e^{-\omega t}\int\limits_{0}^{t}e^{\omega \tau }\left\Vert
\Delta v(\tau )\right\Vert _{L^{2}\left(
\mathbb{R}
^{n}\backslash B\left( 0,r\right) \right) }\left\Vert w\left( \tau \right)
\right\Vert _{H^{1}\left(
\mathbb{R}
^{n}\right) }d\tau +\frac{\widetilde{c}_{3}}{r}e^{-\omega
t}\int\limits_{0}^{t}e^{\omega \tau }\left\Vert w\left( \tau \right)
\right\Vert _{H^{4}\left(
\mathbb{R}
^{n}\right) }d\tau
\end{equation*}%
\begin{equation*}
\leq \widetilde{c}_{4}\left( e^{-\omega t}+\Pi _{r}e^{(\widetilde{c}%
_{1}-\omega )t}\right) \left\Vert \theta _{2}-\theta _{1}\right\Vert
_{H^{2}\left(
\mathbb{R}
^{n}\right) \times L^{2}\left(
\mathbb{R}
^{n}\right) }\text{, \ \ }\forall t\geq 0\text{ and }\forall r\geq 1\text{,}
\end{equation*}%
where
\begin{equation*}
\Pi _{r}:=\underset{t\geq 0}{\sup }\left\Vert \Delta v(t)\right\Vert
_{L^{2}\left(
\mathbb{R}
^{n}\backslash B\left( 0,r\right) \right) }+\frac{1}{r}\text{.}
\end{equation*}%
Then, the last inequality, together with (4.7), gives%
\begin{equation*}
\int\limits_{0}^{t}e^{-\omega \left( t-\tau \right) }\left\Vert S\left( \tau
\right) \theta _{2}-S\left( \tau \right) \theta _{1}\right\Vert
_{H^{2}\left(
\mathbb{R}
^{n}\backslash B\left( 0,r\right) \right) \times L^{2}(%
\mathbb{R}
^{n}\backslash B\left( 0,r\right) )}d\tau
\end{equation*}%
\begin{equation}
\leq \widetilde{c}_{5}\left( e^{-\omega t}+\Pi _{r}e^{(\widetilde{c}%
_{2}-\omega )t}\right) t\left\Vert \theta _{2}-\theta _{1}\right\Vert
_{H^{2}\left(
\mathbb{R}
^{n}\right) \times L^{2}\left(
\mathbb{R}
^{n}\right) },\text{ \ \ \ }\forall t\geq 0\text{ and }\forall r\geq 1\text{.%
}  \tag{4.10}
\end{equation}%
So, by (4.8) and (4.10), we have
\begin{equation*}
\left\Vert S\left( t\right) \theta _{2}-S\left( t\right) \theta
_{1}\right\Vert _{H^{4}\left(
\mathbb{R}
^{n}\right) \times H^{2}\left(
\mathbb{R}
^{n}\right) }\leq \widetilde{c}_{6}\left( e^{-\omega t}+e^{-\omega t}t+\Pi
_{r}e^{(\widetilde{c}_{2}-\omega )t}t\right) \left\Vert \theta _{2}-\theta
_{1}\right\Vert _{H^{4}\left(
\mathbb{R}
^{n}\right) \times H^{2}\left(
\mathbb{R}
^{n}\right) }
\end{equation*}%
\begin{equation}
+\widetilde{c}_{6}\sup_{0\leq \tau \leq t}\left\Vert S\left( \tau \right)
\theta _{2}-S\left( \tau \right) \theta _{1}\right\Vert _{H^{2}\left(
B\left( 0,r\right) \right) \times L^{2}\left( B\left( 0,r\right) \right) },%
\text{ \ \ \ }\forall t\geq 0\text{ and }\forall r\geq 1\text{.}  \tag{4.11}
\end{equation}%
From the compactness of $\mathcal{A}$, it follows that $\Pi _{r}\rightarrow
0 $, uniformly with respect to the trajectories from $\mathcal{A}$ , as $%
r\rightarrow \infty $. Thus, applying [18, Theorem 7.9.6], by (4.7) and
(4.11), we obtain the finite dimensionality of $\mathcal{A}$.
\end{proof}

\end{document}